\documentclass{amsart}
\usepackage{amsmath,amssymb}
\usepackage{xcolor}
\usepackage{verbatim}
\usepackage{tabmac}
\usepackage{hyperref}

\newtheorem{conj}{Conjecture}
\newtheorem{lem}[conj]{Lemma}
\newtheorem{prop}[conj]{Proposition}
\newtheorem{thm}[conj]{Theorem}
\newtheorem{cor}[conj]{Corollary}
\theoremstyle{remark}
\newtheorem{rem}{Remark}
\newtheorem{ex}[conj]{Example}

\numberwithin{equation}{section}

\newcommand{\cat}{\mathrm{cat}}

\newcommand{\C}{\mathbb{C}}
\newcommand{\bi}{\mathbf{i}}
\newcommand{\hbi}{\hat{\bi}}
\newcommand{\ba}{\mathbf{a}}
\newcommand{\hba}{\hat{\ba}}

\newcommand{\bY}{\mathbb{Y}}

\newcommand{\ccat}{\mathrm{ccat}}
\newcommand{\charge}{\mathrm{charge}}

\newcommand{\CT}{\mathrm{CT}}
\newcommand{\dd}{{d^\bullet}}
\newcommand{\qH}{\mathcal{H}}

\newcommand{\heta}{{\hat{\eta}}}
\newcommand{\hmu}{{\hat{\mu}}}
\newcommand{\HL}[2]{{H}^{#1}_{#2}}

\newcommand{\K}[2]{{K}^{#1}_{#2}}
\newcommand{\Knuth}{\equiv}
\newcommand{\la}{\lambda}
\newcommand{\lad}{{\lambda^\bullet}}
\newcommand{\La}{\Lambda}
\newcommand{\LRtab}[2]{\mathrm{LR}^{#1}_{#2}}
\newcommand{\LRword}[1]{\mathrm{LR}_{#1}}
\newcommand{\LR}{\mathrm{LR}}

\newcommand{\M}[2]{M^{#1}_{#2}}

\newcommand{\mud}{{\mu^{\bullet}}}

\newcommand{\ov}{\mathrm{ov}}
\newcommand{\pair}[2]{\langle #1\,,\,#2\rangle}
\newcommand{\qHL}[2]{\mathcal{H}^{#1}_{#2}}
\newcommand{\qK}[2]{\mathcal{K}^{#1}_{#2}}
\newcommand{\qKred}[2]{\overline{\mathcal{K}}^{#1}_{#2}}

\newcommand{\row}{\mathrm{row}}
\newcommand{\sgn}{\mathrm{sgn}}
\newcommand{\shape}{\mathrm{shape}}
\newcommand{\tab}[2]{T^{#1}_{#2}}

\newcommand{\tSd}{\tilde{S}^\bullet}
\newcommand{\tQ}{t_{Q_1}}
\newcommand{\tW}{\tilde{W}}
\newcommand{\Yam}[1]{Y^{#1}}

\newcommand{\Sd}{S^\bullet}

\newcommand{\Td}{T^\bullet}
\newcommand{\word}{\mathrm{word}}
\newcommand{\Xd}{X^\bullet}
\newcommand{\Y}{\mathbb{Y}}
\newcommand{\Z}{\mathbb{Z}}

\newcommand{\tu}{\tilde{u}}
\newcommand{\tv}{\tilde{v}}
\newcommand{\tw}{\tilde{w}}

\newcommand{\tS}{\tilde{S}}

\newcommand{\tg}{\tilde{\gamma}}

\newcommand{\tU}{\tilde{U}}

\newcommand{\uu}{\underline{u}}
\newcommand{\tuu}{\underline{\tu}}

\newcommand{\RSK}{\mathrm{RSK}}

\newcommand{\dom}{\trianglerighteq}

\newcommand{\wt}{\mathrm{wt}}
\newcommand{\id}{\mathrm{id}}

\newcommand{\vn}{\varnothing}
\newcommand{\gd}{{\gamma^\bullet}}
\newcommand{\ds}{\dot{s}}

\newcommand{\ve}{\varepsilon}
\newcommand{\vp}{\varphi}

\newcommand{\tXi}{\tilde{\Xi}}

\newcommand{\Fr}{\mathrm{Frob}}
\newcommand{\tr}{\mathrm{Tr}}
\newcommand{\type}{\tau}

\newcommand{\Sym}{\mathrm{Sym}}
\newcommand{\Ind}{\mathrm{Ind}}

\newcommand{\tM}{\widetilde{M}}

\newcommand{\hp}{\hat{p}}
\newcommand{\Ga}{\Gamma}

\title{On cyclic quiver parabolic Kostka-Shoji polynomials}

\author{Daniel Orr}
\address{
	Department of Mathematics (MC 0123),
	460 McBryde Hall, Virginia Tech,
	225 Stanger St.,
	Blacksburg, VA 24061 USA}
\email{dorr@vt.edu}
\author{Mark Shimozono}
\address{
	Department of Mathematics (MC 0123),
	460 McBryde Hall, Virginia Tech,
	225 Stanger St.,
	Blacksburg, VA 24061 USA}
\email{mshimo@math.vt.edu}

\begin{document}
\maketitle

\begin{abstract} We obtain an explicit combinatorial formula for certain parabolic Kostka-Shoji polynomials associated with the cyclic quiver, generalizing results of Shoji and of Liu and Shoji.
\end{abstract}

\section*{Introduction}
In \cite{OS:newquiver} an analogue of parabolic Hall-Littlewood (HL) symmetric function was defined for general quivers. For the single loop quiver this recovers the parabolic Hall-Littlewood symmetric functions defined in \cite{SW} \cite{SZ}
which in turn generalize the classical modified HL functions denoted $Q'_\mu(X;t)$ in \cite{Mac}. For cyclic quivers this produces a modified (and parabolic) form of Shoji's Hall-Littlewood functions for the complex reflection group $G(r,1,n)$ \cite{Sho1,Sho2,Sho3} in the case of limit symbols. The parabolic Hall-Littlewood functions for general quivers encode the graded multiplicities in $GL$-equivariant Euler characteristics of vector bundles on Lusztig's convolution diagrams \cite{Lu}. Based on a higher vanishing conjecture of \cite{OS:newquiver}, certain quiver Hall-Littlewood functions are expected---and in some cases known by \cite{P}---to expand positively in the tensor Schur basis.

In this article we give a positive combinatorial formula for certain parabolic quiver Hall-Littlewood symmetric functions living on the cyclic quiver (Theorem~\ref{T:main}); the Schur positivity of these functions was not previously known by other means such as \cite{P}. Our formula expresses the Schur expansion coefficients of these functions, i.e., their quiver Kostka-Shoji polynomials, as a sum over certain multitableaux weighted by charge. This generalizes a result of \cite{LiuShoji} in the non-parabolic case of the cyclic quiver with two nodes. The latter was derived from results of \cite{AH} on the intersection cohomology of the enhance nilpotent cone. We give an independent, combinatorial proof of our formula.
For a single node we recover the graded character of tensor products of Kirillov-Reshetikhin (KR) modules for affine $\mathfrak{sl}_n$ or equivalently the Euler characteristic of global sections of a line bundle twist of the cotangent bundle to a partial flag variety \cite {Br} \cite{Shim:cyclageLR} 
\cite{Shim:affinecrystal}.

Our formula implies that the cyclic quiver Hall-Littlewood functions to which it applies are the images of ordinary (single loop) parabolic Hall-Littlewood functions under a plethystic substitution (Corollary~\ref{C:HLpleth}). We interpret this in the representation theory of the wreath product groups $\Ga_n=\Ga^n\rtimes S_n$, where $\Ga$ is a cyclic group of order $r$, by showing that the plethystic substitution is realized by a graded form of induction from $S_n$ to $\Ga_n$ (Proposition~\ref{P:frob-ind}). We deduce that the graded induction of the (singly-graded) Garsia-Procesi module $R_\mu$ \cite{GP} is a graded $\Ga_n$-module whose Frobenius characteristic can be identified with a cyclic quiver parabolic Hall-Littlewood function (Corollary~\ref{C:Rmu-ind}).

In a future work \cite{OS:wreathHL} we will study the relationship between the parabolic Hall-Littlewood functions for the cyclic quiver
and Haiman's wreath Hall-Littlewood functions (wreath Macdonald polynomials at $q=0$) \cite{H}.

\section*{Acknowledgements}
The authors gratefully acknowledge support from NSF grant DMS-1600653.

\section{Statement of Main Results}

\subsection{Cyclic quiver symmetric functions}
\label{SS:cyclic quiver symmetric functions}
In this article we work with the cyclic quiver $Q$ on $r$ nodes.
It has node set $Q_0 = \Z/r\Z$ and arrow set $Q_1 = \{(i,i+1)\mid i\in Q_0\}$ where expressions involving elements of $Q_0$ are understood modulo $r$. We set $\hat{Q}_1=Q_1\setminus\{(r-1,0)\}$.

Let $T^{Q_1} \cong (\C^*)^r$ be an algebraic torus with a copy of $\C^*$ for each arrow $a\in Q_1$. Write $R(T^{Q_1}) = \Z[t_a^{\pm1}\mid a\in Q_1]$ where $t_a$ is the exponential weight of the $a$-th copy of $\C^*$. Let $\La^Q = \bigotimes_{i\in Q_0} \La^{(i)}$ be the tensor power of the algebra of symmetric functions $\La$ with a tensor factor $\La^{(i)}\cong \La$ for each vertex $i\in Q_0$, with coefficients in $R(T^{Q_1})$.
We use notation $f[X^{(i)}]$ for $f\in \La$ and $i\in Q_0$, to denote the tensor in $\La^{Q_1}$ having $f$ in the $i$-th tensor factor and
$1$'s elsewhere. Let $\Y$ be Young's lattice of partitions.
Then $\La^Q$ has a basis given by the tensor Schur functions
$s_\lad = \prod_{i\in Q_0} s_{\la^{(i)}}[X^{(i)}]$ where $\lad=(\la^{(0)},\dotsc,\la^{(r-1)})\in\Y^{Q_0}$ is a $Q_0$-multipartition.

\subsection{Lusztig data}
\label{SS:Lusztig data}
A Lusztig datum is a sequence of triples $\{(i_k,a_k,\mu(k)) \mid 1\le k\le m\}$
where $i_k\in Q_0$, $a_k\in\Z_{>0}$, and $\mu(k)\in X_+(GL_{a_k})$ is a dominant weight. Let
$\bi=(i_1,i_2,\dotsc,i_m)$, $\ba=(a_1,a_2,\dotsc,a_m)$ and
$\mu(\bullet)=(\mu(1),\mu(2),\dotsc,\mu(m))$.

We say the Lusztig datum $(\bi,\ba,\mu(\bullet))$
is \textit{periodic} if $i_m=r-1$ and $i_{k+1}=i_k+1$
for all $1\le k<m$, \textit{even} if $i_1=0$, \textit{Borel} if $a_k=1$ for all $k$,
\textit{rectangular} if the weight $\mu(k)$ is a rectangle for all $k$, that is, has the form $(\gamma_k^{a_k})$ for some $\gamma_k\in \Z_{\ge0}$,
\textit{concentrated at $i\in Q_0$} if $\mu(k)=(0^{a_k})$
for all $k$ such that $i_k\ne i$, and \textit{dominant} if,
for every $i\in Q_0$, the sequence of weights
$\mu(k)$ for $i_k=i$, concatenates to a dominant weight. A periodic Lusztig datum will be called \textit{balanced} if $a_k=a_{k+1}$ whenever $i_k\neq r-1$.

For even periodic Borel Lusztig data, the quiver Kostka-Shoji polynomial was defined by Finkelberg and Ionov \cite{FI}. If all $t_a$ are set to a single variable this recovers Shoji's Green functions for the complex reflection group given by the wreath product of $S_n$ with the cyclic group $\Z/r\Z$, in the case of limit symbols \cite{Sho1,Sho2}.

We consider dominant periodic balanced rectangular Lusztig data that are concentrated at node $r-1$. Such data can be specified by a triple
$(\mu,\eta,i_1)$ where
$\mu\in \Y_s$ is a partition with $s$ parts which are allowed to be zero, $\eta=(\eta_1,\dotsc,\eta_s)\in\Z_{>0}^s$ with $\sum_k\eta_k=n$
(which defines a standard parabolic subgroup of $GL_n$),
and $i_1\in Q_0$. Given $(\mu,\eta,i_1)$ a Lusztig datum is constructed using $s$ passes around the cyclic quiver $Q$ with node set $\Z/r\Z$.
The first pass places $GL_{\eta_1}$-weights at nodes $i_1$, $i_1+1$, up to $r-1$. For $k$ going from $2$ up to $s$, the $k$-th pass places
$GL_{\eta_k}$-weights at all the nodes in order from $0$ to $r-1$.
For $i\ne r-1$, all weights placed at node $i$ are the zero weight.
At node $r-1$ the $k$-th pass places the rectangular weight
$(\mu_k^{\eta_k})$ having $\eta_k$ rows and $\mu_k$ columns.

\begin{ex} Let $r=2$, $i_1=1$, $\eta=(\eta_1,\eta_2)=(2,2)$,
and $\mu=(2,1)$. We have Lusztig datum
with first pass $(1,2,(2,2))$ and second pass
$(0,2,(0,0))$, $(1,2,(1,1))$ so that $\bi=(1,0,1)$,
$\ba=(2,2,2)$ and $\mu(\bullet)=((2,2),(0,0),(1,1))$.
This Lusztig datum is dominant: the weights at $0\in Q_0$ concatenate to $(0,0)\in X_+(GL_2)$ and at $1\in Q_0$ concatenate to $(2,2,1,1)\in X_+(GL_4)$.
\end{ex}

\subsection{Cyclic quiver HL functions and Kostka-Shoji polynomials}
Let $\Omega[X]= \prod_{x\in X} (1-x)^{-1}$ be the Cauchy kernel.
For a symmetric function $f$ let $f[X^{(i)}]^\perp$ be the operator adjoint (with respect to the Hall inner product) to multiplication by $f$, all with respect to the tensor factor $\La^{(i)}$. For $i\in Q_0$ define the generating series of operators \cite{OS:newquiver}:
\begin{align}\label{E:Hop}
H^{(i)}(z) &= \sum_{d\in\Z} z^d H^{(i)}_d = \Omega[zX^{(i)}] \Omega[-z^{-1}(X^{(i)}-t_{i,i+1}X^{(i+1)})]^\perp.
\end{align}
For all $d\in\Z$, $H^{(i)}_d$ is a linear endomorphism of $\La^Q$ of
degree $d$. It is the cyclic quiver analogue of Garsia's variant of Jing's Hall-Littlewood creation operator \cite{G} \cite{J}.

For $i\in Q_0$ and $a\in \Z_{>0}$ let $Z=(z_1,z_2,\dotsc,z_a)$
and $R(Z) = \prod_{1\le k<\ell\le a} (1- z_\ell/z_k)$.
Define \cite{OS:newquiver}:
\begin{align}\label{E:paraHop}
H^{(i,a)}(Z) &= \sum_{\beta\in\Z^a} z^\beta H^{(i,a)}_\beta = 
R(Z) H^{(i)}(z_1) \dotsm H^{(i)}(z_a).
\end{align}
This is the cyclic quiver analogue of the parabolic Hall-Littlewood creation operator of \cite{SZ}.

Given a Lusztig datum $(\bi,\ba,\mu(\bullet))$, the associated cyclic quiver parabolic Hall-Littlewood function $\qHL{(\bi,\ba)}{\mu(\bullet)}[\Xd;t_{Q_1}]\in \La^Q$ is defined by $\qHL{(\vn,\vn)}{\vn} = 1$ and 
\begin{align}
\qHL{(\bi,\ba)}{\mu(\bullet)}[\Xd;t_{Q_1}] = 
\HL{(i_1,a_1)}{\mu(1)} \cdot \qH^{(\hbi,\hba)}_{\hmu(\bullet)}[\Xd;t_{Q_1}]
\end{align}
where $\hbi=(i_2,i_3,\dotsc)$, $\hba=(a_2,a_3,\dotsc)$,
and $\hmu(\bullet) = (\mu(2),\mu(3),\dotsc)$.

The cyclic quiver parabolic Kostka-Shoji polynomials
$\qK{\lad}{\bi,\ba,\mu(\bullet)}(\tQ)\in R(T^{Q_1})$ are defined by
\begin{align}
\qHL{(\bi,\ba)}{\mu(\bullet)}[\Xd;\tQ] = 
\sum_{\lad\in \Y^{Q_0}} 
\qK{\lad}{\bi,\ba,\mu(\bullet)}(\tQ) s_\lad[\Xd].
\end{align}
Since the quiver $Q$ has one cycle, the polynomials
$\qK{\lad}{\bi,\ba,\mu(\bullet)}(\tQ)$ essentially have only one
variable. To make this precise
let $\alpha^{(i)}=\epsilon^{(i)}-\epsilon^{(i+1)}$ for $0\le i\le r-2$
be the simple roots of $GL(\C^{Q_0})$ where $\{\epsilon^{(i)}\mid i\in Q_0\}$ is the standard basis of $\Z^{Q_0}$. By the proof of \cite[Lemma 2.20]{OS:newquiver},
$\qK{\lad}{\bi,\ba,\mu(\bullet)}(\tQ) =0$
unless $\sum_{i\in Q_0} |\la^{(i)}| \epsilon^{(i)} -
\sum_k |\mu(k)| \epsilon^{(i_k)}$ is in the root lattice of
$GL(\C^{Q_0})$, which means it has the form $\sum_{i=0}^{r-2} a_i \alpha^{(i)}$ for some integers $a_i$. In that case we define the Laurent monomial
\begin{align}\label{E:texp}
  t_{\hat{Q}_1}^{\lad-\mu(\bullet)} = \prod_{i=0}^{r-2} t_{i,i+1}^{a_i}.
\end{align}
By \cite[Lemma 2.20]{OS:newquiver} there is a polynomial
$\qKred{\lad}{\bi,\ba,\mu(\bullet)}(t)\in\Z[t]$ called the reduced
Kostka-Shoji polynomial, such that
\begin{align}\label{E:redKS}
\qK{\lad}{\bi,\ba,\mu(\bullet)}(\tQ) = 
t_{\hat{Q}_1}^{\lad-\mu(\bullet)}
\qKred{\lad}{\bi,\ba,\mu(\bullet)}(t_{01} t_{12}\dotsm t_{r-1,0}).
\end{align}
We see that the single essential variable is the product of the arrow variables going around the cycle. 

When $(\bi,\ba,\mu(\bullet))$ is the rectangular Lusztig datum associated to the triple $(\mu,\eta,i_1)$ as in \S \ref{SS:Lusztig data} we use the notation
\begin{align}
\qHL{}{\mu,\eta,i_1}[\Xd;\tQ] &= \qHL{\bi,\ba}{\mu(\bullet)}[\Xd;\tQ] \\
\qK{\lad}{\mu,\eta,i_1}(\tQ) &= \qK{\lad}{\bi,\ba,\mu(\bullet)}(\tQ)
\\
\qKred{\lad}{\mu,\eta,i_1}(t)&=\qKred{\lad}{\bi,\ba,\mu(\bullet)}(t).
\end{align}

\subsection{LR multitableaux} The definitions in the next
several subsections follow \cite{Shim:cyclageLR}. We refer the reader to Appendix~\ref{S:tableau} for conventions and further details on our tableau constructions.
Let $A_1,\dotsc,A_s$ be the consecutive subintervals of $[n]=\{1,2,\dotsc,n\}$ where $|A_k|=\eta_k$ and let $Y_k$
be the rectangular tableau of width $\mu_k$ and height $\eta_k$ whose
rows are constant and contain the values of $A_k$.

\begin{ex} Continuing the previous example we have $A_1=\{1,2\}$, $A_2=\{3,4\}$ and
$$
Y_1 = \tableau[sty]{1&1\\2&2} \qquad Y_2 = \tableau[sty]{3\\4}
$$
\end{ex}

Say that a word $u$ in the alphabet $[n]$ is $(\mu,\eta)$-LR (Littlewood-Richardson) if for all $1\le k\le s$, we have the Knuth equivalence $u|_{A_k}\equiv Y_k$ (see \S \ref{SS:Knuth}) where
$u|_{A_k}$ is the subword obtained by erasing letters of $u$ not in $A_k$. 

We identify (semistandard) tableaux with their reading words (see \S\ref{SS:Knuth}). By a multitableau we mean a tuple $T^\bullet=(T^{(i)})_{i\in Q_0}$ of tableau. The word of a multitableau is defined as the concatenation $\word(T^\bullet)=T^{(0)}T^{(1)}\dotsm T^{(r-1)}$ of reading words of each tableau in the multitableau. Say that a tableau or multitableau is $(\mu,\eta)$-LR if its word is. 

Let $\LRword{\mu,\eta}$ denote the set of $(\mu,\eta)$-LR words.
Let $\LRtab{\la}{\mu,\eta}$ be the set of $(\mu,\eta)$-LR tableaux of shape $\la\in \Y$.

\begin{rem} The name LR tableau comes from the fact
(see Cor. \ref{C:RW}) that the LR coefficient of $s_\la$ in 
the product of the Schur functions $s_{(\mu_k^{\eta_k})}$,
is equal to $|\LRtab{\la}{\mu,\eta}|$.
\end{rem}

\subsection{Rotation of LR words}
Let $w_0^{\eta}$ be the long element of the subgroup
$S_\eta=S_{\eta_1}\times S_{\eta_2}\times \dotsm \times S_{\eta_s}$ of $S_n$. The following operation allows the ``rotation" of LR words.
This must be used every time letters pass between the $0$-th and $(r-1)$-th tableaux in a multitableau.

\begin{prop} \label{P:rotation}
For words $u$ and $v$, 
$uv\in \LRword{\mu,\eta}$ if and only if $(w_0^\eta v)(w_0^\eta u)\in \LRword{\mu,\eta}$, where $w_0^\eta$ acts via the composition of $GL_n$ crystal reflection operators (see \S\ref{SS:crystal}). This induces an action of the cyclic group $\Z/N\Z$ on $\LRword{\mu,\eta}$ where $N=\sum_k \mu_k \eta_k$.
\end{prop}

\begin{rem} 
\begin{enumerate}
\item If $\eta_k=1$ for all $k$ then $S_\eta$ is the identity subgroup and the action is usual rotation of a word by positions.
\item Since the crystal reflection operators $s_i$ are well-defined on Knuth classes (that is, if $u\equiv v$ then $s_i u \equiv s_i v$),
$\Z/N\Z$ acts on the set $\bigsqcup_{\la\in\Y}\LRtab{\la}{\mu,\eta}$. This leads to the cyclage poset structure on LR tableaux \cite{Shim:cyclageLR} which generalizes the cyclage poset on tableaux \cite{LS}.
\end{enumerate}
\end{rem}

\begin{ex} Let $u=2$ and $v=4 2 1 3 1$. Then
	$uv\in \LRword{\mu,\eta}$ because
	$uv|_{A_1} = 2211\equiv Y_1$ and
	$uv|_{A_2} = 43 \equiv Y_2$.
We have $w_0^\eta(u) = 1$, $w_0^\eta v = 4 2 1 3 2$.
We have $(w_0^\eta v)(w_0^\eta u) = 421321\in Y_{\mu,\eta}$ because $421321|_{A_1} = 2121\equiv Y_1$
and $421321|_{A_2}=43\equiv Y_2$.
\end{ex}

\subsection{A charge statistic for LR words}
Let $|u|$ denote the length of the word $u$.

\begin{prop} \label{P:LR-charge} \cite{Shim:cyclageLR}
There is a unique function $\charge_{\mu,\eta}:\LRword{\mu,\eta}\to \Z_{\ge0}$ such that
\begin{enumerate}
\item If $\mu$ and $\eta$ are empty then $\charge_{\mu,\eta}(\vn)=0$
where $\vn$ is the empty word.
\item $\charge_{\mu,\eta}$ is constant on Knuth classes.
\item If $u Y_1 v\in \LRword{\mu,\eta}$ then writing
$\mu=(\mu_1,\hmu)$ and $\eta=(\eta_1,\heta)$, we have
\begin{align*}
\charge_{\mu,\eta}(u Y_1 v) = |v| + \charge_{\hmu,\heta}((w_0^{\heta} v) (w_0^{\heta} u)).
\end{align*}
\end{enumerate}
\end{prop}

\begin{rem} If $\eta_k=1$ for all $k$ (the Lusztig datum is Borel)
then $\charge_{\mu,\eta}$ is the charge statistic of Lascoux and Sch\"utzenberger \cite{LS}.
\end{rem}

\subsection{Tableau formula}

For $\lad\in\Y_n^{Q_0}$ let $\LRtab{\lad}{\mu,\eta,i_1}$ be the
set of multitableaux $\Td$ of shape $\lad$ (i.e., each $T^{(i)}$ has shape $\la^{(i)}$) such that
$\word(\Td)\in \LRword{\mu,\eta}$ and
$T^{(i)}|_{A_1}=\vn$ for $0\le i<i_1$. Note that for $i_1=0$,
$\LRtab{\lad}{\mu,\eta,0}$ is just the set of multitableaux of shape $\lad$ whose word is $(\mu,\eta)$-LR.

For $\Td\in\LRtab{\lad}{\mu,\eta,i_1}$ we set $\charge_{\mu,\eta}(\Td)=\charge_{\mu,\eta}(\word(\Td))$.

\begin{thm} \label{T:main}
We have
\begin{align}\label{E:tableau}
\qKred{\lad}{\mu,\eta,i_1}(t) =
	 \sum_{\Td\in \LRtab{\lad}{\mu,\eta,i_1}} t^{\charge_{\mu,\eta}(\Td)}. 
\end{align}
\end{thm}

\begin{rem}
In the special case of $r=2$ nodes and $\eta_k=1$ for all $k$ (Borel case), an equivalent formula was obtained in \cite{LiuShoji} using results of \cite{AH}. We give a combinatorial proof of Theorem~\ref{T:main} in Section~\ref{S:proof} which is independent of these results.
\end{rem}

\begin{rem}
The higher vanishing criterion of \cite{P} is sufficient to establish the positivity of $\qKred{\lad}{\mu,\eta,i_1}(t)$ in the Borel case for any $r$, but not in general.
\end{rem}

For $\Td\in \LRtab{\lad}{\mu,\eta,i_1}$, 
using notation similar to \eqref{E:texp} let
\begin{align}
  \wt_{\mu,\eta}(\Td) = t_{\hat{Q}_1}^{\lad-N \epsilon^{(r-1)}} (t_{0,1}t_{1,2}\dotsm t_{r-1,0})^{\charge_{\mu,\eta}(\Td)},
\end{align}
so that by \eqref{E:redKS}, Theorem~\ref{T:main} is equivalently expressed as
\begin{align}
\qK{\lad}{\mu,\eta,i_1}(t_{Q_1}) =	 
	 \sum_{\Td\in \LRtab{\lad}{\mu,\eta,i_1}} \wt_{\mu,\eta}(\Td).
\end{align}


For the single node version,
let $\qHL{}{\mu,\eta}[X;t]$ be the parabolic Hall-Littlewood symmetric function associated with the sequence of rectangles $(\mu_k^{\eta_k})$ for $1\le k\le s$ \cite{SZ}. 
Equivalently it is the graded character of a Kirillov-Reshetikhin module \cite{Shim:affinecrystal} and the cyclic quiver Hall-Littlewood symmetric function of \cite{OS:newquiver} for the single loop quiver with Lusztig data $(0,\eta_k,(\mu_k^{\eta_k}))$ for $1\le k\le s$
and $t=t_{00}$. Define their coefficients by
\begin{align}
\qHL{}{\mu,\eta}[X;t] = \sum_{\la\in\Y} \qK{\la}{\mu,\eta}(t) s_\la[X].
\end{align}

\begin{thm} \label{T:one node rectangles}
\cite{Shim:cyclageLR}
\begin{align}
\mathcal{K}^\la_{\mu,\eta}(t) =\sum_{T\in \LRtab{\la}{\mu,\eta}} t^{\charge_{\mu,\eta}(t)} 
\end{align}
\end{thm}

\begin{rem} In \cite{Shim:cyclageLR} the above sum over $T$ was shown to give a graded isotypic component of a line bundle twist of the
coordinate ring of a nilpotent adjoint orbit closure \cite{SW}. Subsequently the parabolic Jing operators were defined \cite{SZ} and the above isotypic component was shown to agree with the coefficient of a Schur function in the parabolic Hall-Littlewood function, which by definition is created by the parabolic Jing operators. The connection of these notions to the Kirillov-Reshetikhin character was proved in \cite{Shim:affinecrystal}; independently similar but dual combinatorics were developed in \cite{SchWar} for KR characters.
\end{rem}

\begin{cor} \label{C:HLpleth}
With
\begin{align}
Y = \sum_{i=0}^{r-1} t_{i,i+1} t_{i+1,i+2}\dotsm t_{r-2,r-1} X^{(i)}
\end{align}
we have
\begin{align}
\qHL{}{\mu,\eta,i_1=0}[\Xd;t_{Q_1}] &=
\qHL{}{\mu,\eta}[Y;t_{01}t_{12}\dotsm t_{r-1,0}].
\end{align}
\end{cor}
\begin{proof} 
Iterating the coproduct formula $\Delta(s_\la)= \sum_{\mu,\nu\in\Y}\pair{s_\la}{s_\mu s_\nu} s_\mu \otimes s_\nu$ and appropriately scaling alphabets we have
\begin{align}
s_\la[Y] &= \sum_{\lad\in \Y^{Q_0}} t_{\hat{Q}_1}^{\lad-N \epsilon^{(r-1)}} c^{\la}_{\lad} s_\lad[\Xd] 
\end{align}
where $c^\la_\lad = \pair{s^\la}{\prod_{i\in Q_0} s_{\la^{(i)}}}$.
Applying this to $\qHL{}{\mu,\eta}[X;t]$ we have
\begin{align*}
  \qHL{}{\mu,\eta}[Y;t_{01}\dotsm t_{r-1,0}] &=
  \sum_{\lad\in\Y^{Q_0}} s_\lad[\Xd] \,t_{\hat{Q}_1}^{\lad-N\epsilon^{(r-1)}} \sum_{\la\in\Y}
  c^\la_{\lad} \qK{\la}{\mu,\eta}(t_{01}\dotsm t_{r-1,0}).
\end{align*}
Consider the map $\Td\mapsto \uu \mapsto (P(\uu),Q(\uu))$
where $\uu$ is the sequence of row words obtained from the
rows of $T^{(r-1)}$, then the rows of $T^{(r-2)}$,
up to the rows of $T^{(0)}$ and $(P(\uu),Q(\uu))$ is the column insertion Robinson-Schensted-Knuth tableau pair (see \S \ref{SS:column insertion}). 
By Theorem \ref{T:White} and Cor. \ref{C:RW} and their notation,
the image consists of all tableau pairs
$(P,Q)$ such that $Q$ is $D$-compatible where $D$ is the skew shape 
$\la^{(r-1)}*\dotsm*\la^{(0)}$. Restricting this bijection
to $(\mu,\eta)$-LR words, we obtain a bijection
\begin{align}
\LRtab{\lad}{\mu,\eta,i_1=0} &\to \bigsqcup_{\la} \LRtab{\la}{\mu,\eta} \times \{ \text{$D$-compatible tableaux of shape $\la$} \}.
\end{align}
Since Knuth equivalence preserves $\charge_{\mu,\eta}$ the result follows from Theorems \ref{T:main} and \ref{T:one node rectangles}.
\end{proof}

\section{Graded representations of wreath products}

In this section we discuss the meaning of our results in the representation theory of the wreath product group $\Ga_n=\Ga^n\rtimes S_n$ where $\Ga$ is a cyclic group of order $r$. We show in particular that the plethystic substitution of Corollary~\ref{C:HLpleth} is, on the level of Frobenius characteristics, a type of graded induction from $S_n$ to $\Ga_n$. It follows that the induction of the Garsia-Procesi module $R_\mu$ \cite{GP} results in a graded $\Ga_n$-module whose Frobenius characteristic is equal, up to the involution $\omega$ on symmetric functions, to a cyclic quiver parabolic Hall-Littlewood function.


\newcommand{\st}{\mathsf{t}}

We assume in this section that $t_{i,i+1}=\st$ for all $i\in Q_0$, with $\st$ serving as the grading parameter. (Note the the difference between this $\st$ and $t$ above which played the role of the product of all arrow variables $t=t_{0,1}t_{1,2}\dotsm t_{r-1,0}$.)

\subsection{Frobenius characteristics of graded $S_n$-modules}
Let $R(S_n)$ denote the Grothendieck ring of the category of graded $S_n$-modules $M=\oplus_{d\ge 0} M_d$ with $M$ finite-dimensional over $\C$. The graded Frobenius characteristic map
\begin{align*}
M\mapsto \Fr(M;\st)&=\frac{1}{n!}\sum_{w\in S_n}\tr_M(w;\st)p_{\type(w)}
\end{align*}
gives a linear isomorphism of $R(S_n)$ with degree $n$ symmetric functions over $\Z[\st]$.
Here we set $\tr_M(x;\st)=\sum_d \st^d\, \mathrm{\tr}_{M_d}(x)$ for any graded endomorphism $x$ of a finite-dimensional graded vector space $M=\oplus_{d\ge 0} M_d$. For a permutation $w\in S_n$, we let $\type(w)$ denote its cycle type, i.e., the partition of $n$ with parts given by the lengths of cycles in $w$. And for $\la\in\Y$, $p_\la=p_{\la_1}p_{\la_2}\dotsm$ is the corresponding power sum basis element in $\La$.

\subsection{The ring $\La^\Ga$}
We realize $\Ga$ as the group of $r$-th roots of unity and fix a generator $\zeta\in\C^\times$. Accordingly, we let $\La^\Ga$ denote the tensor power $\La^{\otimes Q_0}$ of symmetric functions with coefficients in $\C[\st]$. Let $\La^\Ga_n\subset\La^\Ga$ denote the $\C[\st]$-submodule of elements of total degree $n$.

There are two natural sets of power sum generators in $\La^\Ga$, one indexed by conjugacy classes in $\Ga$ and the other by irreducible representations of $\Ga$. In our present situation, we abuse notation and identify both index sets with $Q_0$. Our previously introduced $p^{(i)}_s$ for $i\in Q_0$ and $s\ge 0$ specializes under $t_{i,i+1}\equiv\st$ to the power sum in $\La^\Ga$ indexed by the irreducible $\Ga$-representation $\zeta\mapsto \zeta^i$; all previously introduced notation, such as the alphabets $X^{(i)}$, will refer to this set of generators. The $s$-th power sum indexed by the conjugacy class $\{\zeta^i\}$ will be denoted $\hp^{(i)}_s$. We have
\begin{align*}
\La^\Ga = \C[\st]\big[p^{(i)}_s : i\in Q_0, s\ge 0\big] = \C[\st]\big[\hp^{(i)}_s : i\in Q_0, s\ge 0\big]
\end{align*}
with $\deg p^{(i)}_s = \deg \hp^{(i)}_s = s$. The two sets of generators are related by the Fourier transform on $\Ga$ as follows:
\begin{align}\label{E:power-sum-fourier}
p^{(j)}_s=r^{-1}\sum_{i=0}^{r-1} \zeta^{ij}\hp^{(i)}_s.
\end{align}
For $\la^\bullet\in\Y^{Q_0}$ we set
\begin{align*}
p_{\la^\bullet}&=\prod_{i\in Q_0} p^{(i)}_{\la^{(i)}},\qquad
\hp_{\la^\bullet}=\prod_{i\in Q_0} \hp^{(i)}_{\la^{(i)}}
\end{align*}
where $p^{(i)}_{\la^{(i)}}=p^{(i)}_{\la^{(i)}_1}p^{(i)}_{\la^{(i)}_2}\dotsm$ as usual, and similarly for $\hp^{(i)}_{\la^{(i)}}$.

\subsection{Conjugacy in $\Ga_n$}
Suppose $g=(g_1,\dotsc,g_n,w)\in \Ga_n$. For a cycle $C=(ab \dotsm c)$ of $w\in S_n$, we write $g_C=g_ag_b\dotsm g_c$ and say that $C$ has color $i$ (with respect to $g$) if $g_C=\zeta^i$ where $0\le i<r$. We define the cycle type $\type^\bullet(g)$ to be the $Q_0$-tuple of partitions of total size $n$ given as follows: $\type^{(i)}(g)$ lists the sizes of cycles of $w$ having color $i$. This gives a bijection between conjugacy classes in $\Ga_n$ and $Q_0$-tuples of partitions of total size $n$.

\subsection{Frobenius characteristics of graded $\Ga_n$-modules}
Let $R(\Ga_n)$ be the complexified Grothendieck ring of the category of graded $\Ga_n$-modules $M=\oplus_{d\ge 0}M_d$ with $M$ finite-dimensional over $\C$.
In this setting, the Frobenius characteristic map
\begin{align*}
M &\mapsto \Fr^\Gamma(M;\st)=|\Ga_n|^{-1}\sum_{g\in \Gamma_n}\tr_M(g;\st)\hp_{\type^\bullet(g)}
\end{align*}
gives an isomorphism $R(\Ga_n)\cong\La^\Ga_n$ sending irreducible $\Ga_n$-modules to the tensor Schur functions (relative to the $p^{(i)}_s$) of degree $n$ \cite[I, Appendix B]{Mac}.

\subsection{Graded induction}

Let $S=\Sym(\C^n)=\C[x_1,\dotsc,x_n]$, with $\Ga_n$ acting by the natural algebra automorphisms given by $g\cdot x_k= g_{w(k)} x_{w(k)}$ for $g$ as above. We take the standard grading on $S$ with $\deg x_k=1$. Consider the $\Ga_n$-submodule
\begin{align}\label{E:pol-r-bounded}
S_{<r}&=\bigoplus_{0\le m_1,\dotsc,m_n<r}\C x_1^{m_1}\dotsm x_n^{m_n}.
\end{align}
Note that $S_{<r}\cong\C[\Ga^n]$ as ungraded $\Ga^n$-modules, with an isomorphism given by
\begin{align}\label{E:S-Ga-iso}
x_1^{m_1}\dotsm x_n^{m_n} &\mapsto \sum_{g=(g_1,\dotsc,g_n)\in\Ga^n}g_1^{-m_1}\dotsm g_n^{-m_n}\cdot g.
\end{align}

For any finite-dimensional graded $S_n$-module $M$ we define a graded $\Ga_n$-module $\Ind(M)$ as the following tensor product of $\Ga_n$-modules:
\begin{align}\label{E:ind-def}
\Ind(M)=S_{<r}\otimes_{\C} \tM.
\end{align}
Here $\tM=M$ as ungraded vector spaces, with the $\Ga_n$-action on $\tM$ trivially extended from that of $S_n$ on $M$; we choose the grading on $\tM$ which is dilated from that of $M$ by a factor of $r$, i.e., $\tM=\oplus_{d\ge 0}\tM_{rd}$ where $\tM_{rd}=M_d$. As ungraded $\Ga_n$-modules we clearly have $\Ind(M)\cong \Ind_{S_n}^{\Ga_n}(M)$.

\begin{rem}\label{R:ind-I}
In the case when $M=S/I$ is a graded $S_n$-module determined by an $S_n$-equivariant homogeneous ideal $I\subset S$, we have the following equivalent description of $\Ind(M)$. Let $J$ be the image of $I$ under the $r$-th power ring homomorphism given by $x_k\mapsto x_k^r$ on the generators of $S$. 
Then $SJ\subset S$ is a $\Ga_n$-equivariant homogeneous ideal and $\Ind(M)\cong S/SJ$ as graded $\Ga_n$-modules.
\end{rem}

\begin{prop}\label{P:frob-ind}
For any finite-dimensional graded $S_n$-module $M$ we have
\begin{align}\label{E:frob-ind}
\Fr^\Ga(\Ind(M);\st) &= \Fr(M;\st^r)[X^{(0)}+\st X^{(1)}+\dotsm+\st^{r-1}X^{(r-1)}].
\end{align}
\end{prop}

\begin{proof}
We compute $\Fr^\Ga(\Ind(M);\st)$ directly.
By \eqref{E:ind-def} we have
\begin{align*}
\tr_{\Ind(M)}(g;\st)&=\tr_{S_{<r}}(g;\st)\tr_{\tM}(w;\st)\\
&=\tr_{S_{<r}}(g;\st)\tr_{M}(w;\st^r).
\end{align*}
for any $g=(g_1,\dotsc,g_n,w)\in\Ga_n$.

Now fix such a $g$ and let $C_1,\dotsc,C_\ell$ be the cycles of $w$ (in any fixed order) and $i_1,\dotsc,i_\ell$ their colors (with respect to $g$). By considering the matrix of $g$ with respect to the basis of $S_{<r}$ given by the monomials in \eqref{E:pol-r-bounded}, one finds\footnote{One may be tempted to compute $\tr_{S_{<r}}(g;\st)$ via the standard basis of $\C[\Ga^n]\cong S_{<r}$. However, this basis is not homogeneous with respect to \eqref{E:S-Ga-iso}. The ungraded computation immediately yields $\tr_{\C[\Ga^n]}(g;1)=\hp_{\tau(w)}^{(0)}$, which is seen to equal $p_{\tau(w)}[X^{(0)}+X^{(1)}+\dotsm+X^{(r-1)}]$ by \eqref{E:power-sum-fourier}.}
\begin{align*}
\tr_{S_{<r}}(g;\st) 
&=\sum_{0\le j_1,\dotsc, j_\ell<r}\zeta^{i_1j_1+\dotsm+i_\ell j_\ell}\st^{j_1|C_1|+\dotsm+j_\ell|C_\ell|}
\end{align*}
where $|C|$ denote the length of a cycle $C$.
Each summand above is a diagonal entry corresponding to a basis element of the form $x_{C_1}^{j_1}\dotsm x_{C_\ell}^{j_\ell}$, where $x_C=x_ax_b\dotsm x_c$ for a cycle $C=(ab\dotsm c)$; the other diagonal entries vanish.
Hence
\begin{align*}
r^{-n}\sum_{g_1,\dotsc,g_n\in\Gamma}\tr_{S_{<r}}(g;\st)\hp_{\type^\bullet(g)}
&=r^{-n}r^{(|C_1|-1)+\dotsm+(|C_\ell|-1)}\sum_{0\le i_1,\dotsc,i_\ell<r}\tr_{S_{<r}}(g;\st)\hp_{\type(g)}\\
&= \prod_{s=1}^\ell \sum_{0\le j_s<r}\sum_{0\le i_s <r} r^{-1}\zeta^{i_sj_s}\st^{j_s|C_s|}\hp_{|C_s|}^{(i_s)}\\
&= \prod_{s=1}^\ell \sum_{0\le j_s<r}\st^{j_s|C_s|} p_{|C_s|}^{(j_s)}\\
&= p_{\type(w)}[X^{(0)}+\st X^{(1)}+\dotsm+\st^{r-1}X^{(r-1)}].
\end{align*}
From this we immediately deduce \eqref{E:frob-ind}.
\end{proof}

\subsection{Induction of $R_\mu$ as a quiver Hall-Littlewood function} For a partition $\mu$ of size $n$, consider the ordinary parabolic Hall-Littlewood function $\mathcal{H}_{(1,\dotsc,1),\mu}$ indexed by the sequence of rectangles of sizes $\mu_k\times 1$ (i.e., columns of height $\mu_k$). By \cite[Example 2.3(2)]{SW}, we have $\mathcal{H}_{(1,\dotsc,1),\mu}=\omega\,\Fr(R_\mu;\st)$ where $R_\mu=S/I_\mu$ is the graded $S_n$-module of \cite{GP} and $\omega$ is the involution on $\La$ such that $\omega(s_\la)=s_{\la'}$ for the transpose $\la'$ of $\la$. 

By abuse of notation, we use $\omega$ also to denote the involution on $\La^\Ga$ given by the tensor power $\omega^{\otimes Q_0}$. Combining Corollary~\ref{C:HLpleth} and Proposition~\ref{P:frob-ind} we see that the Frobenius characteristic of $\Ind(R_\mu)$ is given by a cyclic quiver parabolic Hall-Littlewood function as follows:
\begin{cor}\label{C:Rmu-ind}
For any partition $\mu$ of size $n$ we have
\begin{align}
\omega \,\Fr^\Ga(\Ind(R_\mu);\st)
=
\mathcal{H}_{(1,\dotsc,1),\mu,i_1=0}[X^\bullet;t_{i,i+1}\equiv \st].
\end{align}
\end{cor}

\begin{rem}
Generators for the defining ideal of $\Ind(R_\mu)$ are obtained simply by replacing $x_1,\dotsc,x_n$ in the the Tanisaki generators of $I_\mu$ \cite[I.5]{GP} by their $r$-th powers, thanks to Remark~\ref{R:ind-I}. In the case of $\mu=(1^n)$, $R_n=R_{(1^n)}$ and $\Ind(R_n)$ are the coinvariant algebras of $S_n$ and $\Ga_n$, respectively. Stembridge \cite[Theorem 6.6]{St} gives the graded $\Ga_n$-module structure of $\Ind(R_n)$. Chan and Rhoades \cite{CR} extend this to a generalized coinvariant algebra for $\Ga_n$ depending on an additional parameter $k$ such that $0\le k\le n$.
\end{rem}

\section{Proof of Theorem \ref{T:main}}\label{S:proof}

\subsection{Recurrence}
We recall a recurrence for the parabolic Kostka-Shoji polynomials
\cite[\S 4]{OS:newquiver}.

Let $(\bi,\ba,\mu(\bullet))$ denote the Lusztig data for $\mu,\eta,i_1$. So $\bi=(i_1,i_1+1,i_1+2,\dotsc)$,
$\ba$ consists of $\eta_1$ repeated $r-i_1$ times followed
by $\eta_k$ repeated $r$ times for $2\le k\le s$,
and $\mu(\bullet)$ is the sequence
starting with $r-i_1-1$ copies of $(0^{\eta_1})$ followed by
$(\mu_1^{\eta_1})$ and then for $2\le k\le s$,
$r-1$ copies of $(0^{\eta_k})$ followed by $(\mu_k^{\eta_k})$.

Let $(i_1,a_1,\mu(1))=(i,a,\nu)$. We have $\nu=(\mu_1^{\eta_1})$
if $i=r-1$ and $\nu=(0^{\eta_1})$ otherwise.
For $w\in S_n$ define $\alpha(w)\in X(GL_a)=\Z^a$ and
$\beta(w) \in X(GL_{n-a})=\Z^{n-a}$ by
\begin{align}\label{E:alphabeta}
w^{-1}(\la^{(i)}+\rho)-\rho = (\alpha(w),\beta(w))
\end{align}

Let $S_n^a$ be the set of minimum length coset representatives
for $S_n/(S_a \times S_{n-a})$. By \cite{OS:newquiver} we have
\begin{align}\label{E:rec not last no kostka}
\qK{\lad}{\mu,\eta,i_1}(\tQ) &=
\sum_{w\in S_n^a} (-1)^w t_{i,i+1}^{\alpha(w)} \sum_{\gd}
c^{\gamma^{(i+1)}}_{\la^{(i+1)},\alpha(w)} \qK{\gd}{\mu,\eta,i_1+1}(\tQ)
\end{align}
where $\gamma^{(j)}=\la^{(j)}$ for $0\le j < i$,
$\gamma^{(i)}=\beta(w)$, and $\gamma^{(i+1)}\in \Y_n$.

Let $\K{\gamma/\la}{\sigma}$ be the Kostka number,
the number of semistandard tableaux of shape $\gamma/\la$ and
weight $\sigma$. Using the Jacobi-Trudi formula for the skew
Schur function we have

\begin{align}\label{E:rec not last}
\qK{\lad}{\mu,\eta,i_1}(t_{Q_1}) &=
\sum_{\substack{w\in S_n \\ \gd \in \Y_n^{Q_0}}}
(-1)^w t_{i,i+1}^{|\alpha(w)|} \K{\gamma^{(i+1)}/\la^{(i+1)}}{\alpha(w)}
\K{\gamma^{(i)}}{\beta(w)} \qK{\gd}{\mu,\eta,i_1+1}(t_{Q_1})
\end{align}
where 
\begin{align}\label{E:gammalambda}
\gamma^{(j)}= \la^{(j)} \qquad \text{for $j\in Q_0\setminus\{i,i+1\}$.}
\end{align}

If $i_1=r-1$, the formula and proof is very similar.
We have $\nu=(\mu_1^{\eta_1})$ and obtain
\begin{align}\label{E:rec last}
\qK{\lad}{\mu,\eta,r-1}(t_{Q_1}) &=
\sum_{\substack{w\in S_n \\ \gd \in \Y_{n-a}^{Q_0}}}
(-1)^w t_{r-1,0}^{|\alpha(w)/\nu|} \K{\gamma^{(0)}/\la^{(0)}}{\alpha(w)-\nu}
\K{\gamma^{(r-1)}}{\beta(w)} \qK{\gd}{\hmu,\heta,0}(t_{Q_1})
\end{align}
such that \eqref{E:gammalambda} holds.

\begin{rem} The recurrence \eqref{E:rec not last no kostka} is more efficient than \eqref{E:rec last} but the cancelling bijection is much simpler to define for the latter.
\end{rem}

\subsection{Morris data}
Let $\tab{\gamma/\la}{\nu}$ be the set of semistandard tableaux
of shape $\gamma/\la$ and weight $\nu$.

A $(\lad;\mu,\eta,i_1)$-Morris datum is a tuple $\Xi=(w, \Sd, U^{(i)},U^{(i+1)})$ where
$w\in S_n$ is such that $\alpha(w) \ge \nu$ componentwise, $U^{(i)}\in \tab{\gamma^{(i)}}{\beta(w)}$, $U^{(i+1)}\in \tab{\gamma^{(i+1)}/\la^{(i+1)}}{\alpha(w)-\nu}$, and $\Sd\in \LRtab{\gd}{\mu,\eta,i_1+1}$ if $0\le i_1<r-1$ 
and $\Sd\in \LRtab{\gd}{\hmu,\heta,0}$ if $i_1=r-1$ 
with $\gd$ as in \eqref{E:gammalambda}. Define the sign and weight of $\Xi=(w,\Sd,U^{(i)},U^{(i+1)})$ by
\begin{align}
\label{E:signMorris}
\sgn(\Xi) &= (-1)^{\ell(w)} \\
\label{E:weightMorris}
\wt_{Q_1}(\Xi) &=
t_{i,i+1}^{|\alpha(w)/\nu|}
\times
\begin{cases}
\wt_{\mu,\eta}(\Sd) & \text{if $0\le i_1<r-1$}\\
\wt_{\hmu,\heta}(\Sd) & \text{if $i_1=r-1$.}
\end{cases}
\end{align} 

Write $\M{\lad}{\mu,\eta,i_1}$ for the set of $(\lad;\mu,\eta,i_1)$-Morris data. It suffices to find a sign-reversing weight-preserving involution $\Phi$
on $\M{\lad}{\mu,\eta,i_1}$ whose fixed point set has a weight-preserving bijection with $\LRtab{\lad}{\mu,\eta,i_1}$.

\subsection{Embedding LR multitableaux into Morris data}
We define a map 
\begin{align}
\LRtab{\lad}{\mu,\eta,i_1} &\overset{\iota}{\longrightarrow}  \M{\lad}{\mu,\eta,i_1} \\
\iota(\Td)&=(\id,\Sd,U^{(i)},U^{(i+1)})
\end{align}
as follows. 

Suppose first that $0\le i_1<r-1$.
Let $u_k$ be the $k$-th row of $T^{(i)}$, denoted $\row_k(T^{(i)})$,
for $1\le k\le a$
and let $S^{(i)}$ be the last $n-a$ rows of $T^{(i)}$ so that $T^{(i)} \equiv S^{(i)} u_a\dotsm u_1$.
Let $\Psi(u_a \dotsm u_1, T^{(i+1)}) = (S^{(i+1)},U^{(i+1)})$ where $\Psi$ is defined in \S \ref{SS:column insertion} via Schensted column insertion.
Let $S^{(j)}=T^{(j)}$ for $j\in Q_0\setminus\{i,i+1\}$. Then
\begin{align*}
T^{(0)} \dotsm\dotsm T^{(r-1)} &\equiv
S^{(0)} \dotsm S^{(i)} u_a\dotsm u_1 T^{(i+1)} S^{(i+2)}\dotsm S^{(r-1)} \\
&\equiv S^{(0)} \dotsm S^{(i)} S^{(i+1)}\dotsm S^{(r-1)}
\end{align*}
so that $\Sd$ is $(\mu,\eta)$-LR. For $0\le j<i$ we have
$S^{(j)}|_{A_1}=T^{(j)}|_{A_1}=\vn$ since $\Td\in\LRtab{\lad}{\mu,\eta,i_1}$. $S^{(i)}|_{A_1}=\vn$ because
$A_1$ consists of the smallest $a$ values and such values must occur in the first $a$ rows of a tableau of partition shape such as $T^{(i)}$.
These values are removed from the $i$-th tableau and put into the $(i+1)$-th tableau. It follows that $\Sd\in \LRtab{\gd}{\mu,\eta,i_1+1}$
where $\gd=\shape(\Sd)$.

If $i_1=r-1$ there are two additional issues: there is a copy of $Y_1$ that must be removed, and the action of the cyclic group on $(\mu,\eta)$-words must be employed to preserve the LR property. Since $\Td$ is $(\mu,\eta)$-LR and has no letters of $A_1$ except in $T^{(r-1)}$ it follows
that $T^{(r-1)}\supset Y_1$. Let $u_a \dotsm u_2 u_1$
be the first $a$ rows and $S^{'(r-1)}$ be the last
$n-a$ rows of $T^{(r-1)}$. Then for $1\le k\le a$, $u_k = k^{\mu_1} v_k$ for a row word (weakly increasing word) $v_k$. We have
\begin{align*}
T^{(0)} T^{(1)} \dotsm T^{(r-1)} &\equiv
  T^{(0)} T^{(1)} \dotsm T^{(r-2)} S^{'(r-1)} Y_1 v_a \dotsm v_1.
\end{align*}
This word is $(\mu,\eta)$-LR. Let us ``rotate" this word by $|v_a|+\dotsm + |v_1|$ positions to the right. By Lemma \ref{L:skew tableau preserving} there are tableaux
$T^{'(0)},S^{(1)},\dotsc,S^{(r-2)}$, $S^{(r-1)}$ of the same shapes
as $T^{(0)},\dotsc,T^{(r-2)},S^{'(r-1)}$ and row words $v'_k$ with $|v'_k|=|v_k|$ for $1\le k\le a$, such that
\begin{align*}
w_0^\eta (T^{(0)} T^{(1)} \dotsm T^{(r-2)} S^{'(r-1)} Y_1) &=
T^{'(0)} S^{(1)}\dotsm S^{(r-2)} S^{(r-1)} Y_1 \\
w_0^\eta (v_a \dotsm v_1) &= v'_a \dotsm v'_1.
\end{align*}
Define $\Psi(v'_a\dotsm v'_1,T^{'(0)})=(S^{(0)},U^{(0)})$. Then
\begin{align*}
v'_a \dotsm v'_1 T^{'(0)} S^{(1)}\dotsm S^{(r-1)} Y_1\equiv
S^{(0)}S^{(1)}\dotsm S^{(r-1)} Y_1
\end{align*}
is $(\mu,\eta)$-LR, and therefore
$\Sd\in\LRtab{\gd}{\hmu,\heta,0}$ for $\gd=\shape(\Sd)$. 

In either case define $U^{(i)}=\Yam{\shape(S^{(i)})}$ to be the tableau of shape $\shape(S^{(i)})$ filled with letter $k$ in each row $k$.
This completes the definition of $\iota$.

\begin{lem}\label{L:embedding charge}
For $\Td\in \LRtab{\lad}{\mu,\eta,i_1}$ we have
\begin{align}
\wt_{\mu,\eta}(\Td) &= \wt_{Q_1}(\iota(\Td)).
\end{align}
\end{lem}
\begin{proof} Suppose $0\le i_1<r-1$. Here $\nu=\vn$. For $u=u_a\dotsm u_1$ as in the definition of $\iota$ we have $|u|=|\alpha(w)/\nu|$. Also $\word(\Td)\equiv \word(\Sd)$ so that
$\charge_{\mu,\eta}(\Td)=\charge_{\mu,\eta}(\Sd)$.
Letting $\gd=\shape(\Sd)$ and $N=\sum_{k=1}^s \mu_k\eta_k$ we have
\begin{align*}
\sum_{j\in Q_0} (|\la^{(j)}|-|\gamma^{(j)}|) \epsilon^{(j)} &=
|\alpha(w)/\nu| (\epsilon^{(i)}-\epsilon^{(i+1)}) 
\end{align*}
from which we deduce that $t_{\hat{Q}_1}^{\gd-N\epsilon^{(r-1)}} t_{i,i+1}^{|\alpha(w)/\nu|} = 
t_{\hat{Q}_1}^{\lad-N\epsilon^{(r-1)}}$
as required.

Suppose $i_1=r-1$. 
Let $\hat{N}=\sum_{k=2}^s \mu_k\eta_k$. We have $|\nu|=\mu_1\eta_1$, $|u|=|\alpha(w)/\nu|$, $|\la^{(r-1)}|-|\gamma^{(r-1)}|= |u| + \mu_1\eta_1$, $|\la^{(0)}| - |\gamma^{(0)}| = -|u|$ so that
\begin{align*}
-N\epsilon^{(r-1)}+\sum_{j\in Q_0} |\la^{(j)}| \epsilon^{(j)} &=
-\hat{N}\epsilon^{(r-1)}+\sum_{j\in Q_0} |\gamma^{(j)}|\epsilon^{(j)} + |\alpha(w)/\nu| (\epsilon^{(r-1)}-\epsilon^{(0)}).
\end{align*}
By Proposition~\ref{P:LR-charge} we have $\charge_{\mu,\eta}(\Td)=\charge_{\hmu,\heta}(\Sd)+|u|$. Therefore
\begin{align*}
t_{\hat{Q}_1}^{\lad-N\epsilon^{(r-1)}} &= 
t_{\hat{Q}_1}^{\gd-\hat{N}\epsilon^{(r-1)}} 
(t_{0,1}\dotsm t_{r-2,r-1})^{-|\alpha(w)/\nu|} \\
&= t_{r-1,0}^{|\alpha(w)/\nu|} t_{\hat{Q}_1}^{\gd-\hat{N}\epsilon^{(r-1)}}  (t_{0,1}\dotsm t_{r-1,0})^{-|\alpha(w)/\nu|} \\
&=  t_{r-1,0}^{|\alpha(w)/\nu|} t_{\hat{Q}_1}^{\gd-\hat{N}\epsilon^{(r-1)}}  (t_{0,1}\dotsm t_{r-1,0})^{\charge_{\hmu,\heta}(\Sd)-\charge_{\mu,\eta}(\Td)}
\end{align*}
as required.
\end{proof}

\subsection{Cancellation} Let $\Xi=(w,\Sd,U^{(i)},U^{(i+1)})\in \M{\lad}{\mu,\eta,i_1}$. We define a sign-reversing weight-preserving involution $\Phi$ on $\M{\lad}{\mu,\eta,i_1}$ with fixed point set the image of $\iota$. Let $\Phi(\Xi)=\tXi= (\tw,\tSd,\tU^{(i)},\tU^{(i+1)})$ with $\tXi$ to be specified.

The cancellation begins by trying to find a $\iota$-preimage for $\Xi$.

Suppose $0\le i<r-1$. Let
\begin{align}
\Psi^{-1}(S^{(i)},U^{(i)}) &=(u_n\dotsm u_{a+1},\vn) \\
\Psi^{-1}(S^{(i+1)},U^{(i+1)})&=(u_a\dotsm u_1,T^{(i+1)}).
\end{align}

Say that $\Xi$ has a violation if $u_n\dotsm u_1$ is not the row word factorization of a tableau of partition shape.
By Proposition~\ref{P:QYam}, this is equivalent to saying that the tableau $Q(\uu)$ of \S \ref{SS:column insertion} is not Yamanouchi. If $w\ne\id$ then $Q(\uu)$ cannot be Yamanouchi as its weight $(\alpha(w),\beta(w))$ is not dominant.

Suppose $\Xi$ does not have a violation. Then $w=\id$,
$u_n\dotsm u_1$ is the word of a tableau, say, $T^{(i)}$,
and $u_a\dotsm u_1$ and $u_n\dotsm u_{a+1}$ are also tableau words.
This implies that $U^{(i)}$ and $U^{(i+1)}$ are Yamanouchi
(but $U^{(i+1)}$ has skew shape).

Let $T^{(j)}=S^{(j)}$ for $j\in Q_0\setminus\{i,i+1\}$. 

Then $\iota(\Td)=(w,\Sd,U^{(i)},U^{(i+1)})$. Define $\Phi(\Xi)=\Xi$.

Otherwise suppose $\Xi$ has a violation. 
Take the rightmost letter $p+1$ in $\word(Q(\uu))$
that is $p$-unpaired. We define $\tw=ws_p$ and
\begin{align}
  \Psi(\tuu,\vn) &= (P(\uu), s_pe_p(Q(\uu))) \\
  \Psi(\tu_n\dotsm \tu_{a+1},\vn) &= (\tS^{(i)},\tU^{(i)}) \\
  \Psi(\tu_a \dotsm \tu_1, T^{(i+1)}) &= (\tS^{(i+1)}, \tU^{(i+1)})
\end{align}
and $\tS^{(j)}=S^{(j)}$ for $j\in Q_0\setminus\{i,i+1\}$.
Since $\word(\tSd)\equiv P(\uu)\equiv \word(\Sd)$ and $s_pe_p(Q(\uu))$ has weight $\tw(\lambda^{(i)}+\rho)-\rho$, it follows readily that $\tXi\in \M{\lad}{\mu,\eta,i_1}$ and 
$\wt_{Q_1}(\tXi)=\wt_{Q_1}(\Xi)$. Note that $\tXi$ has a violation due to Lemma~\ref{L:se involution}, with the rightmost $p$-unpaired letter $p+1$ playing the same role in $\word(Q(\tuu))$ as it did in $\word(Q(\uu))$. The latter ensures that the map $\Xi\to\tXi$ is an involution.


This defines $\tXi$ and establishes its properties for the case $0\le i_1<r-1$.

Suppose $i_1=r-1$. Let
\begin{align*}
\Psi^{-1}(S^{(0)},U^{(0)}) = (v_a\dotsm v_1, W^{(0)}).
\end{align*}
We have the $(\hmu,\heta)$-LR word
\begin{align*}
S^{(0)} S^{(1)} \dotsm S^{(r-1)} &\equiv
v_a \dotsm v_1 W^{(0)} S^{(1)} \dotsm S^{(r-1)}. 
\end{align*}
Let $v'_a \dotsm v'_1$ be the sequence of row words with $|v_k|=|v'_k|$ for $1\le k\le a$, and $W^{'(0)}$, $S^{'(1)}$, $\dotsc$, $S^{'(r-2)}$, $S^{'(r-1)}$ the tableaux of the same shapes as $W^{(0)}$, $S^{(1)}$, $\dotsc$, $S^{(r-1)}$ such that 
\begin{align*}
w_0^\eta(v_a\dotsm v_1)&= v'_a\dotsm v'_1 \\
w_0^\eta(W^{(0)} S^{(1)} \dotsm S^{(r-2)} S^{(r-1)}) &=
W^{'(0)} S^{'(1)} \dotsm S^{'(r-2)} S^{'(r-1)}.
\end{align*}
Then the word
\begin{align*}
W^{'(0)} S^{'(1)} \dotsm S^{'(r-1)} v'_a \dotsm v'_1
\end{align*}
is $(\hmu,\heta)$-LR. Let $u'_k = k^{\mu_1} v'_k$ for $1\le k\le a$.
Let $$\Psi^{-1}(S^{'(r-1)},U^{(r-1)})=(u'_n\dotsm u'_{a+1},\vn).$$
We arrive at $\uu'=(u'_n,\dotsc,u'_1)$. 

Say that $\Xi$ has a violation if $u_n'\dotsm u_1'$ is not the row word factorization of a tableau of partition shape.

Suppose $\Xi$ does not have a violation. As before,
$w=\id$, $u'_n\dotsm u'_1$ is the reading word of a tableau $T^{(r-1)}$. We have $T^{(r-1)}=P(u'_n\dotsm u'_1)=P(S^{'(r-1)}u'_a\dotsm u'_1)$. Letting $T^{(0)}=W^{'(0)}$,
$T^{(j)}=S^{'(j)}$ for $1\le j\le r-2$ we have
\begin{align*}
T^{(0)} T^{(1)} \dotsm T^{(r-2)} T^{(r-1)}  u'_a \dotsm u'_1&\equiv
W^{'(0)} S^{'(1)} \dotsm S^{'(r-2)} S^{'(r-1)} Y_1 v'_a \dotsm v'_1
\end{align*}
which is $(\mu,\eta)$-LR. As before we verify that $\iota(\Td)=\Xi$ and declare that $\Phi(\Xi)=\Xi$.

Suppose $\Xi$ has a violation. The index $p$ and $\tw$ are defined as before. Define
\begin{align}
\Psi(\tuu',\vn) &= (P(\uu'), s_pe_p(Q(\uu'))) \\
\Psi(\tu_n'\dotsm \tu_{a+1}',\vn) &= (\tS^{'(r-1)},\tU^{(r-1)}).
\end{align}

\begin{lem} For $1\le k \le a$, we have $\tu_k' = k^{\mu_1} \tv_k'$ for row words $\tv_k'$.
\end{lem}
\begin{proof} Rather than acting on the $Q$ tableaux 
by $s_pe_p$ we pass between the sequences of row words
$u'_n\dotsm u'_1$ and $\tu'_n\dotsm \tu'_1$ 
by Proposition \ref{P:dual crystal jeu}, acting directly on the
sequences of row words using the dual crystal operator $s_p^* e_p^*$. See also Example \ref{X:jeu}.
We have $\tu_{p+1}'\tu_p'\equiv u_{p+1}u_p$
and $\tu_j'=u_j'$ for $j\notin \{p,p+1\}$. There is nothing to prove
if $p>a$. Suppose that $1\le p<a$. We have
$u'_{p+1} u'_p = (p+1)^{\mu_1} v'_{p+1} p^{\mu_1} v'_p$
where $v'_p$ and $v'_{p+1}$ contain only letters greater than $a$.
Consider the two row skew tableau with lower row $u'_{p+1}$ and upper row $u'_p$, having maximum overlap (that is, with the rows partially slid over each other so as to make a skew tableau using the minimum total number of columns). 

Suppose a letter $p$ moves from row $p$ to row $p+1$ during the computation of $s_p^* e_p^*$. This can only happen if the two rows
start in the same column. This would be a two row
tableau of partition shape, which we need never encounter when
computing $s_p^*e_p^*$, yielding a contradiction.
A letter $p+1$ can never move from row $p+1$ to row $p$:
to do so, right before the $p+1$ moves up the tableau must look like
$$
\tableau[sty]{\dotsm&p&\bullet&\dotsm\\\dotsm&q&q&\dotsm}
$$
where we write $q$ for $p+1$ and $\bullet$ for the moving ``hole" in the jeu-de-taquin. This is a contradiction to semistandardness since
there are exactly $\mu_1$ $p$'s (all in the upper row to the left of the hole) and $\mu_1$ $(p+1)$'s (all in the lower row with one of them located just below the hole) and all other letters are greater than $p+1$.

The remaining case is $p=a$. One must show that
in passing from $u'_{p+1} u'_p$ to $\tu'_{p+1} \tu'_p$
no letter $p$ goes from the $p$-th row to the $(p+1)$-th.
The proof is as in the previous case.
\end{proof}

Let 
\begin{align}
\label{E:sameW}
\tW^{'(0)}&=W^{'(0)} \\
\label{E:sameS}
\tS^{'(j)}&=S^{'(j)}\qquad\text{for $1\le j\le r-2$.}
\end{align}

For $1\le k\le a$ let $\tv_k$ be the row word such that
$|\tv_k|=|\tv'_k|$ and let $\tW^{(0)}$, $\tS^{(1)}$,$\dotsc$,
$\tS^{(r-1)}$ be the tableaux of the same shapes
as $\tW^{'(0)}$, $\tS^{'(1)}$, $\dotsc$, $\tS^{'(r-1)}$ such that 
\begin{align}
w_0^\eta(\tv_a' \dotsm \tv_1') &= \tv_a \dotsm \tv_1 \\
w_0^\eta(\tW^{'(0)}\tS^{'(1)}\dotsm \tS^{'(r-2)} \tS^{'(r-1)}) &=
\tW^{(0)} \tS^{(1)} \dotsm \tS^{(r-1)}.
\end{align}
Let
\begin{align}
\Psi(\tv_a \dotsm \tv_1, \tW^{(0)}) &= (\tS^{(0)}, \tU^{(0)})
\end{align}
It is straightforward to check that 
$\tXi=(\tw,\tSd,\tU^{(r-1)},\tU^{(0)})\in \M{\lad}{\mu,\eta,r-1}$. Again, $\tXi$ has a violation by Lemma~\ref{L:se involution} and by our choice of $p$ the map $\Xi\to\tXi$ is an involution.

We now verify that $\wt_{Q_1}$ is preserved.
Let $v=v_a \dotsm v_1$ and $\tv=v'_a\dotsm v'_1$. We have
\begin{align*}
\charge_{\hmu,\heta}(\Sd) + |v| &= 
\charge_{\mu,\eta}(W^{'(0)}S^{'(1)}\dotsm S^{'(r-1)} v') \\
&= \charge_{\mu,\eta}(\tW^{'(0)}\tS^{'(1)}\dotsm \tS^{'(r-1)} \tv') \\ 
&= \charge_{\hmu,\heta}(\tSd) + |\tv'|
\end{align*}
\begin{align*}
|\gamma^{(r-1)}| &= |\shape(S^{(r-1)})| = |\shape(S^{'(r-1)})| \\
&= \sum_{k=a+1}^n |v'_k| = |Q(\uu')| - \sum_{k=1}^a |v'_k|.
\end{align*}
Similarly $|\tg^{(r-1)}| = |Q(\tu')| - \sum_{k=1}^a |\tv'_k|$.
But $|Q(\uu')| = |Q(\tu')|$ so
$|\gamma^{(r-1)}|-|\tg^{(r-1)}| = |\tv|-|v|$. Similarly
$|\gamma^{(0)}|-|\tg^{(0)}| = |v|-|\tv|$.
One now readily verifies $\wt_{\hmu,\heta}(\Sd)=\wt_{\hmu,\heta}(\tSd)$.

\section{On the catabolizable tableau conjecture of \cite{OS:newquiver}}
In \cite{OS:newquiver} a tableau conjecture was given for
the Kostka-Shoji polynomial for any quiver such that every vertex has in-degree at most one and out-degree at most one.
This applies to the cyclic quiver.
We explain how this conjecture holds when the above conditions
intersect with the conditions of Theorem \ref{T:main}.
In this section we assume the Lusztig data is even, periodic, Borel,
but not necessarily concentrated at node $r-1$. In particular let $\eta_k=1$ for all $i$. Such Lusztig data is parametrized by
$\mud\in\Y_n^{Q_0}$: $\bi$ is $(0,1,2,\dotsc,r-1)$ repeated $n$ times,
$\ba$ is $1$ repeated $rn$ times, and $\mu(\bullet)$ is the sequence of single row partitions of the following sizes:
\begin{align*}
\mu_1^{(0)},\mu_1^{(1)},\dotsc,\mu_1^{(r-1)},
\mu_2^{(0)},\mu_2^{(1)},\dotsc,\mu_2^{(r-1)},\dotsc,
\mu_n^{(0)},\mu_n^{(1)},\dotsc,\mu_n^{(r-1)}.
\end{align*}
Denote by $\qHL{}{\mud}[\Xd,\tQ]$ the corresponding cyclic quiver Hall-Littlewood function and by $\qK{\lad}{\mud}(\tQ)$
its coefficient at $s_\lad$, with $\qKred{\lad}{\mud}(t)$
defined analogously.

\subsection{Single row catabolism}
For a word (or tableau or multitableau) $u$ let $|u|_k$ be the number of times the letter $k$ appears in $u$. For a tableau $T$ let $\hat{T}$ denote the tableau with its first row
(denoted $\row_1(T)$) removed. Similarly for $\la\in\Y$ let $\hat{\la}=(\la_2,\la_3,\dotsc)$ be the partition $\la$ with its first part removed.

Given a multitableau $\Td$, a node $i\in Q_0$ and nonnegative integer $p$, we say that $\Td$ admits $\cat^{(i)}_p$ if 
$|\row_1(T^{(i)})|_1 \ge p$.
Suppose this is so: let $\row_1(T^{(i)})=1^p v$ for a row word $v$.
For $r>1$ define $\cat^{(i)}_p(\Td)$ to be the multitableau
$\Sd$ with $S^{(j)}=T^{(j)}$ for $j\in Q_0\setminus \{i,i+1\}$,
$S^{(i)}=\hat{T}^{(i)}$, and $T^{(i+1)}=P(v T^{(i+1)})$ where
$P$ is the Schensted $P$-tableau (see \S \ref{SS:Knuth}).
For $r=1$ we have a single tableau, $i=0$, and 
$\cat^{(0)}_p(T^{(0)}) = P(v \hat{T}^{(0)})$.

\begin{ex} Let $r=2$, $\lad=((6,4,2,0),(3,2,1,0))$, $\mu=(5, 5, 4, 4)$, and $\Td\in T(\lad,\mu)$ given by
	$$
	\Td = \tableau[sty]{1&1&2&2&3&3\\2&2&4&4\\4&4} \otimes \tableau[sty]{1&1&1\\2&3\\3}
	$$
	Any multitableau admits $\cat_0^{(0)}$. Applying this to $\Td$,
	the first row is removed from the $0$-th tableau and is column inserted into the $1$-th tableau.
	$$
	\tableau[sty]{2&2&4&4\\4&4} \otimes \tableau[sty]{1&1&2&2&3&3} \cdot \tableau[sty]{1&1&1\\2&3\\3} \equiv
	$$
	$$
	\tableau[sty]{2&2&4&4\\4&4} \otimes  \tableau[sty]{1&1&1&1&1&3&3&3\\2&2&2\\3}
	$$
	The resulting multitableau admits $\cat_5^{(1)}$.
	The first row is removed from the $1$-th tableau, 5 of its ones are removed, and the rest of the word is column inserted into the $0$-th tableau:
	$$
	\tableau[sty]{3&3&3} \cdot \tableau[sty]{2&2&4&4\\4&4} \otimes  \tableau[sty]{2&2&2\\3}\equiv
	$$
	$$
	\tableau[sty]{2&2&3&4&4&4&4\\3&3} \otimes  \tableau[sty]{2&2&2\\3} = \cat_5^{(1)} \cat_0^{(0)}(\Td).
	$$
\end{ex}

Let $\Td$ be a multitableau of shape $\lad$ and $d^\bullet\in\Z_{\ge0}^{Q_0}$ an effective dimension vector. The $d^\bullet$-\textit{cascading catabolism} of $\Td$ is defined by
\begin{align}\label{E:big ccat}
\ccat_{d^\bullet}(\Td) = 
\cat_{d^{(r-1)}}^{(r-1)} \cat_{d^{(r-2)}}^{(r-2)}\dotsm \cat_{d^{(1)}}^{(1)} \cat_{d^{(0)}}^{(0)}(\Td).
\end{align}
Of course this need not be defined as the requisite ones may not be present. 

The multipartition $\mud\in\Y_n^{Q_0}$ can also be regarded as a sequence of $n$ effective dimension vectors $\mu_1^\bullet, \mu_2^\bullet,\dotsc,\mu_n^\bullet$.
Say that the multitableau $\Td$ is $\mud$-cascade catabolizable if $|T^\bullet|_k=|\mu_k^\bullet|$ for all $1\le k\le n$ and $\Td$ admits the composition of operators
$\ccat_{\mu^\bullet_n} \dotsm \ccat_{\mu^\bullet_2} \ccat_{\mu^\bullet_1}$ (which will necessarily produce the empty multitableau when applied to $T^\bullet$). Here $\ccat_{\mu^\bullet_k}$ is acting on a multitableau containing letters $k$ through $n$
and is removing letters $k$.

Let $\CT(\lad,\mud)$ be the set of $\mud$-cascade catabolizable multitableaux of shape $\lad$.

\begin{conj} \label{CJ:tab} \cite{OS:newquiver}
\begin{align}
	\qKred{\lad}{\mud}(t) = \sum_{T\in \CT(\lad,\mud)} t^{\charge(\Td)}
\end{align}
\end{conj}

\begin{thm} \label{T:tab} 
Conjecture \ref{CJ:tab} holds
when $\mud$ has the special form $(\emptyset^{r-1},\mu)$ for some $\mu\in\Y_n$, that is, $\mu^{(i)}$ is empty for $i\in Q_0\setminus\{r-1\}$ and $\mu^{(r-1)}$ is the partition $\mu$.
\end{thm}

In \S \ref{SS:catproof} we explain how Theorem
\ref{T:tab} follows from Theorem \ref{T:main}.

\begin{rem} \label{R:monomial for dom}
In the situation of Theorem \ref{T:tab} 
\begin{align}\label{E:monomial}
	t_{\hat{Q}_1}^{\lad-\mud} = \prod_{i=0}^{r-2} t_{i,i+1}^{|\la^{(0)}|+|\la^{(1)}|+\dotsm+|\la^{(i)}|}.
\end{align}
\end{rem}

The single-row rectangle special case of Corollary \ref{C:HLpleth} is:

\begin{cor} For $\mud=(\emptyset^{r-1},\mu)$
\begin{align}
	\qKred{\lad}{\mud}(t) = 
	\sum_{\la} c^\la_{\lad} K_{\la\mu}(t)
\end{align}	
where  $K_{\la\mu}(t)$ is the Kostka-Foulkes polynomial.
\end{cor}

\begin{rem}
In the special case that $\la^{(i)}$ is also empty for $i\in Q_0\setminus\{r-1\}$ so that $\lad= (\emptyset^{r-1},\la)$ for some $\la\in\Y_n$, we have $t_{\hat{Q}_1}^{\lad-\mud}=1$ and the reduced Kostka-Shoji polynomial is the usual Kostka-Foulkes polynomial $K_{\la,\mu}(t)$.
This is a theorem of Shoji \cite{Sho3}.
\end{rem}

\subsection{Proof of Theorem \ref{T:tab}}
\label{SS:catproof}

For a multitableau $\Td$ and positive integer $k$ define
the dimension vector
\begin{align}\label{E:mdef}
	m_k^\bullet(\Td) = \sum_{i\in Q_0} |T^{(i)}|_k \, \epsilon^{(i)}.
\end{align}
It remembers how many letters $k$ there are at the various vertices of a multitableau.

For dimension vectors $\dd,f^\bullet\in \Z^{Q_0}$ define $\dd\dom_{Q_0} f^\bullet$ if $\dd-f^\bullet \in \bigoplus_{i=0}^{r-2} \Z_{\ge0} \alpha^{(i)}$.

\begin{lem} \label{L:ccat} Let $|\Td|_1 = |\dd|$.
Then $\Td$ admits $\ccat_{\dd}$ if and only if $m_1^\bullet(\Td) \dom_{Q_0} \dd$. 
\end{lem}
\begin{proof} To check this it is enough to assume that $\Td$ consists of only ones. Let $T^{(i)}$ consist of $e^{(i)}$ ones so that
$m_1^\bullet(\Td)=e^\bullet$. 
Define $f\in \Z^{Q_0}$ by $f(i) = \sum_{j=0}^i (e^{(j)} - d^{(j)})$. By assumption, $f(r-1)=0$.

By induction on $i$, one proves that
$\cat_{d^{(i-1)}}^{(i-1)}\dotsm \cat_{d^{(0)}}^{(0)}(\Td)$ is defined
if and only if $f(j) \ge 0$ for $0\le j\le i-1$ and in that case,
the resulting multitableau consists of empty tableaux at nodes $0$ through $i-1$, $f(i-1)+e^{(i)}$ ones at node $i$, and $e^{(j)}$ ones at node $j$ for $i<j\le r-1$.
By induction the statement holds at $i=r$, in which case it says that $\Td$ admits $\ccat_{\dd}$ if and only if $e^\bullet \dom_{Q_0} d^\bullet$, in which case the
result of $\ccat_{\dd}(\Td)$ is the empty multitableau.
\end{proof}

\begin{lem} \label{L:remove all at end} Suppose $|\Td|_1=x$.
	\begin{enumerate}
		\item[(a)] $\Td$ admits $\ccat_{(0,\dotsc,0,x)}$.
		\item[(b)] Given $\dd\in\Z_{\ge0}^{Q_0}$ let $x=\sum_{i\in Q_0} d^{(i)}$. Suppose $\Td$ admits $\ccat_{\dd}$. Then $\Td$ admits $\ccat_{(0,\dotsc,0,x)}$ and
		\begin{align}
			\ccat_{(0,0,\dotsc,x)}(\Td) = 
			\ccat_{\dd}(\Td).
		\end{align}
	\end{enumerate}
\end{lem}
\begin{proof} For (a) define $d^{(i)} = |T^{(i)}|_1$ for $i\in Q_0$.
	Then $x=|\Td|_1=|\dd|$ and $\dd \dom_{Q_0} (0,\dotsc,0,x)$.
	By Lemma \ref{L:ccat}, $\Td$ admits $\ccat_{(0,\dotsc,0,x)}$.
	
	Now let $\Td$ and $\dd$ be as in (b). By (a) $\Td$ admits $\ccat_{(0,\dotsc,0,x)}$.
	For a fixed $i\in Q_0\setminus\{r-1\}$ let
	\begin{align*}
		D^\bullet &= \cat^{(i)}_0 \dotsm \cat^{(1)}_0 \cat^{(0)}_0(\Td)\\
		C^\bullet&=\cat^{(i)}_{d^{(i)}} \dotsm \cat^{(1)}_{d^{(1)}} \cat^{(0)}_{d^{(0)}}(\Td).
	\end{align*}
	One may show by induction on $i$ that
	\begin{align*}
		D^{(j)} &=\begin{cases}
			C^{(j)} &\text{for $j\ne i+1$} \\
			P(1^{d^{(0)}+d^{(1)}+\dotsm+d^{(i)}} C^{(i+1)})&\text{for $j=i+1$.}
		\end{cases}
	\end{align*}
	that is, $D^{(i+1)}$ is obtained from $C^{(i+1)}$
	by putting $d^{(0)}+d^{(1)}+\dotsm+d^{(i)}$ more ones in the first row.
	Now consider $i=r-2$.
	The final operators, $\cat^{(r-1)}_x$ and $\cat^{(r-1)}_{d^{(r-1)}}$
	of $\ccat_{(0,\dotsc,0,x)}$ and $\ccat_{\dd}$ respectively, remove the first rows of $D^{(r-1)}$ and $C^{(r-1)}$ respectively, leaving behind the remaining tableaux $\hat{D}^{(r-1)}$
	and $\hat{C}^{(r-1)}$ which now match,
	remove all ones from the first rows, making these row words equal,
	and then inserting this common row word into the equal $0$-th tableaux $C^{(0)}=D^{(0)}$, resulting in the same multitableau at the end.
\end{proof}

\begin{rem}\label{R:ccat n=1}
If $n=1$ then $\CT(\lad,\mud)\neq\emptyset$ if and only if
$\la_1^\bullet \dom_{Q_0} \mu_1^\bullet$. In that case $\CT(\lad,\mud)$ is a singleton, the unique multitableau of shape $\lad$ containing only ones.
\end{rem}

\begin{rem} Lemmas \ref{L:ccat} and \ref{L:remove all at end} can be combined to give an alternative condition to $\mud$-cascade catabolizability. $\Td$ is $\mud$-cascade-catabolizable if and only if $m_1^\bullet(\Td)\dom_{Q_0} \mu_1^\bullet$,
$m_2^\bullet(\ccat_{(0,\dotsc,0,|\mu_1^\bullet|)}(\Td)) \dom_{Q_0} \mu_2^\bullet$, etc.
\end{rem}

Let $T(\lad,\mu)$ be the set of multitableaux of shape $\lad$
and weight $\mu$.

\begin{lem} \label{L:cat vacuous}
For $\mud=(\emptyset^{r-1},\mu)$ we have
$\CT(\lad,\mud)=T(\lad,\mu)$.
\end{lem}
\begin{proof} This follows from Lemma \ref{L:remove all at end}.
\end{proof}

Since $\LR^{\lad}_{\mu,\eta,0}$ clearly also equals $T(\lad,\mu)$ in the Borel case, we see that Theorem~\ref{T:tab} is a consequence of Theorem~\ref{T:main} and Lemma~\ref{L:cat vacuous}.


\appendix

\section{Tableau constructions}\label{S:tableau}

\subsection{Knuth equivalence}
\label{SS:Knuth}

Knuth equivalence $\Knuth$ on words is the transitive closure of the following relations, where $u$ and $v$ are words and $x,y,z$ are values satisyfing
\begin{align*}
u xzy v &\equiv u zxy v &\qquad&\text{for $x\le y<z$} \\
u yxz v &\equiv u yzx v &\qquad&\text{for $x<y\le z$.}
\end{align*}

\begin{lem} \label{L:Knuthres}
Let $u$ and $v$ be words. If $u\equiv v$ then
for all intervals $I$ $u|_I\equiv v|_I$.
\end{lem}

The Ferrers diagram $D(\la)$ of a partition $\la\in\bY$ is the set
of matrix-style pairs $(i,j)\in \Z^2$ with $j\le \la_i$.
By a tableau we mean a semistandard tableau of some partition shape $\la$, a function $T:D(\la)\to \Z_{>0}$ which weakly increases along rows ($T(i,j)\le T(i,j+1)$ for $(i,j),(i,j+1)\in D(\la)$) and 
strictly increases down columns ($T(i,j)< T(i+1,j)$ for $(i,j),(i+1,j)\in D(\la)$).

The reading word (denoted $\word(T)$) of a tableau $T$ is the word obtained by reading the rows of $T$ from left to right, starting with the bottom row and proceeding to earlier rows. We regard a tableau $T$ as a word in this manner.

\begin{ex} 
	\begin{equation*}
	T = \tableau[sy]{1&1&2\\2&3}\qquad\word(T)=23112.
	\end{equation*}
\end{ex}

\begin{thm} \label{T:P tableau}
For every word $u$ (in symbols $1,2,\dotsc$) there is a unique tableau denoted $P(u)$, such that $u\equiv \word(P(u))$.
\end{thm}

$P(u)$ can be computed by Schensted's row or column insertion algorithms \cite{Sch}.

\subsection{Column insertion and RSK}
\label{SS:column insertion}

For $\la,\nu\in\Y$ we say that $\nu/\la$ is a horizontal strip
if $D(\nu)\supset D(\la)$ and the set difference $D(\nu)-D(\la)$
(which we denote by $\nu/\la$) has at most one box in each column.

\begin{prop} \label{P:column insert row}\cite{Sch}
Given $p\in\Z_{\ge0}$ and $\la\in\Y$,
there is a bijection $\psi_\la$ sending $(u,T)$ to $P$
where $u$ is a row word with $|u|=p$,
$T$ is a tableau of shape $\la$,
and $P$ is a tableau of some shape $\nu$ such that
$\nu/\la$ is a horizontal strip of size $p$.
It is uniquely specified by the condition
$uT\equiv P$.
\end{prop}

The forward map is $(u,T)\mapsto P(uT)$. We denote the inverse map by $\psi_\la^{-1}(P)=(u,T)$. These may be computed directly using Schensted column insertion and its reverse \cite{Sch}.

Given a tableau $T$ let $T|_{\le k}$ be the subtableau of $T$
consisting of the entries of value at most $k$. Iterating the above Lemma yields the following bijection.

\begin{prop}\label{P:RSK} \cite{Sch}
There is a unique bijection (the column insertion Robinson-Schensted-Knuth correspondence)  $$(u_n,u_{n-1},\dotsc,u_1)\overset{\RSK}\mapsto (P,Q)$$
from sequences of $n$ row words to pairs of (semistandard) tableaux (with $Q$ on the alphabet $[n]=\{1,2,\dotsc,n\}$) such that for all $0\le k\le n$,
$\shape(Q|_{\le k}) = \shape(P(u_k \dotsm u_2 u_1))$. In particular $|Q|_k = |u_k|$ for all $1\le k\le n$.
\end{prop}

We write $\uu\mapsto (P(\uu),Q(\uu))$ for this bijection.

More generally suppose $T$ is a semistandard tableau of partition shape and $\uu=(u_n,\dotsc,u_1)$ is a sequence of row words.
Define $\Psi(\uu,T)=(P,U)$ where $P=P(u_n\dotsm u_1 T)$
and $U$ is the semistandard skew tableau 
defined by the sequence of partitions
$\shape(P(u_k\dotsm u_2 u_1 T))$ for $0\le k\le n$.
$\Psi$ defines a bijection between pairs $(\uu,T)$
and $(P,Q)$ where $P$ is a semistandard tableau of partition shape
and $Q$ is a semistandard skew tableau with $\shape(Q) = \shape(P)/\shape(T)$ such that $\wt(Q) = (|u_1|,|u_2|,\dotsc,|u_n|)$.

The following is a reformulation of a theorem of D. White \cite{Wh}.

Let $y_\la$ be the Yamanouchi tableau of shape $\la$,
the unique tableau of shape and weight $\la$.
For partitions $\mu\subset\la\in\Y_n$ say that the tableau
$Q$ is $\la/\mu$-compatible if $y_\la \equiv Q y_\mu$.

\begin{thm} \label{T:White}
Let $\mu,\la\in\Y_n$ with $\mu\subset \la$
and let $\uu=u_n\dotsm u_1$ be a sequence of row words with
$|u_i| = \la_i-\mu_i$. Then $\uu$ is the sequence of rows of a semistandard tableau of shape $\la/\mu$ if and only if, for any
tableau $T$ of partition shape, if $\Psi(\uu,T)=(P,Q)$ then
$Q$ is $\la/\mu$-compatible.
\end{thm}

\begin{cor} \label{C:White} For any $\nu\in\Y$, $c^\la_{\mu,\nu}=\pair{s_\nu}{s_{\la/\mu}}$ is the number of
	$\la/\mu$-compatible tableaux of shape $\nu$.
\end{cor}

\begin{cor} \label{C:RW} \cite{RW}
For a sequence of partitions $\la^{(0)},\la^{(1)},\dotsc,\la^{(r-1)}$,
the product of Schur functions $s_{\la^{(0)}}\dotsm s_{\la^{(r-1)}}$ is a skew Schur function, associated with the skew shape
$D=\la^{(r-1)}*\dotsm*\la^{(0)}$
obtained by placing the partitions $\la^{(0)}$ up to $\la^{(r-1)}$ from northeast to southwest.
Thus for $\la\in\Y$, $c^\la_{\lad} = \pair{s_\la}{s_{\la^{(0)}}\dotsm s_{\la^{(r-1)}}}$ is equal to the number of 
$D$-compatible tableaux of shape $\la$.
\end{cor}

\begin{ex} \label{X:star shape}
The skew shape $(2,1) * (2,2) * (3)$ is pictured below.
$$
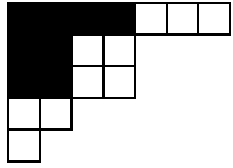
\tableau[sty]{\fl&\fl&\fl&\fl&&&\\ \fl&\fl&&\\ \fl&\fl&&\\ & \\ \\}
$$
\end{ex}

\subsection{$A_{n-1}$ crystal graphs}
\label{SS:crystal} 

Words in $\{1,2,\dotsc,n\}$ of a fixed length,
form a type $A_{n-1}$ crystal graph. 
For such a word $u$, view each $i$ (resp. $i+1$) in $u$ as a right (resp. left) parenthesis. After matching parentheses,
the unpaired parentheses form a subword $i^\phi (i+1)^\epsilon$. Define $\vp_i(u)=\phi$ and $\ve_i(u)=\epsilon$.
If $\ve_i(u)>0$ then $e_i(u)$ is defined by replacing
the unpaired subword by $i^{\phi+1} (i+1)^{\epsilon-1}$. 
If $\ve_i(u)=0$ then $e_i(u)$ is undefined.
If $\vp_i(u)>0$ then $f_i(u)$ is defined by replacing
the unpaired subword by $i^{\phi-1} (i+1)^{\epsilon+1}$. 
If $\vp_i(u)=0$ then $f_i(u)$ is undefined.
The crystal reflection operator $s_i$ acts on a word $u$ by 
replacing the above unpaired subword by $i^{\epsilon} (i+1)^\phi$.
The crystal graph is the directed graph (with edges colored by $1\le i\le n-1$) having a directed arrow colored $i$ from each $u$ to $f_i(u)$. For a fixed $i$ the components of the graph are directed paths called $i$-strings. Along an $i$-string all the letters other than $i$ and $i+1$ and also all the $i$-paired letters, remain the same: only the substring of $i$-unpaired letters changes.
$\ve_i(u)$ (resp. $\vp_i(u)$) is the distance from $u$ to the beginning (resp. end) of its $i$-string. 
If reference to $i$ is needed we will say that letters are $i$-paired or $i$-unpaired and so on.

\begin{rem}\label{R:skew tableau preserving}
The set of (semistandard) tableaux of a fixed (partition) shape $\la\in\Y$ with entries in $1,2,\dotsc,n$ have a type $A_{n-1}$-crystal graph structure
induced by inclusion into the crystal graph of words by sending a tableau to its row-reading word:
the crystal operators $e_i$ and $f_i$ stabilize 
the set of \textit{tableau words}, those which are the reading words of tableaux. More generally the crystal operators preserve the set of tableaux of a fixed skew shape. The reason for these is that the crystal operators preserve the recording tableau $Q(u)$ of a word $u$ and the condition that a word $u$ be the reading word of a tableau of a fixed skew shape, is equivalent to saying that $Q(u)$ belongs to a given set of tableaux depending only on the skew shape \cite{Wh}.
\end{rem}

\begin{lem} \label{L:skew tableau preserving}
Let $D_1,\dots, D_s$ be skew shapes and
$T_1,\dotsc,T_s$ be semistandard tableaux with $T_j$ of shape $D_j$. Let $s_i$ be a crystal reflection operator.
Let $s_i(T_1\dotsm T_s)=u_1\dotsm u_s$ where $|u_j|=|T_j|$ for all $j$. Then $u_j$ is the reading word of a semistandard tableau of shape $D_j$ for all $j$.
\end{lem}
\begin{proof} This follows from Remark \ref{R:skew tableau preserving}
since the word $T_1\dotsm T_s$ is the reading word of a tableau of a skew shape $D_s*\dotsm D_2*D_1$ obtained by placing $D_1,D_2,\dotsc,D_s$ going from southwest to northeast on disjoint sets of northwest-southeast diagonals; see Example \ref{X:star shape}.
\end{proof}

\begin{rem} \label{R:undom weight admits e}
$\ve_i(u)>0$ if $|u|_{i+1}>|u|_i$ since some $i+1$ must be $i$-unpaired.
\end{rem}

\begin{lem}\label{L:se involution}
$\ds_i=s_ie_i$ is an involution on the set of words $u$ such that
$\ve_i(u)>0$. It restricts to an involution on the set of tableaux $T$ of a given shape with $\ve_i(T)>0$.
\end{lem}


\subsection{Dual crystal graph structure on $n$-tuples of row words}
\label{SS:dualcrystal}
Consider the set of $n$-tuples of row words $\uu=(u_n,\dotsc,u_2,u_1)$.
It has a type $A_{n-1}$ crystal graph structure defined by acting on the column insertion RSK $Q$ tableau. That is, we define
\begin{align*}
\ve_i^*(\uu) &= \ve_i(Q(u)) \\
\vp_i^*(\uu) &=\vp_i(Q(u)). 
\end{align*}
For any $\uu$, $s_i^*(\uu)$ is defined by
\begin{align}
P(s_i^*(\uu)) &= P(\uu) \\
Q(s_i^*(\uu)) &= s_i(Q(\uu)).
\end{align}
If $\ve_i^*(\uu)>0$ then $e_i^*(\uu)$ is defined by
\begin{align}
P(e_i^*(\uu)) &= P(\uu) \\
Q(e_i^*(\uu)) &= e_i(Q(\uu)).
\end{align}

Given row words $v$ and $u$ define $\ov(v,u)$ to be the
maximum number of columns $c$ such that $v'u'$ is the word
of a tableau of shape $(c,c)$ where $v'$ (resp. $u'$) is the subword of the last $c$ letters of $v$ (resp. first $c$ letters of $u$).

Say that a word $u$ in the alphabet $\{1,2,\dotsc,n\}$ is Yamanouchi (resp. almost Yamanouchi) if $\ve_i(u)=0$ for all $1\le i\le n-1$ (resp. $2\le i\le n-1$.)

\begin{prop} \label{P:QYam}
Let $\uu=u_n\dotsm u_1$ be a sequence of row words.
Then the number of $i$-pairs in $Q(\uu)$ is equal to 
$\ov(u_{i+1},u_i)$. In particular, $\uu$ is a tableau word if and only if $Q(\uu)$ is Yamanouchi, and $u_n\dotsm u_2$ is a tableau word if and only if $Q(\uu)$ is almost Yamanouchi.
\end{prop}
\begin{proof} Follows from Theorem \ref{T:White}.
\end{proof}

\begin{prop} \label{P:dual crystal jeu}
Let $\uu=u_n\dotsm u_1$ be an $n$-tuple of row words
and let $\uu'=u_n'\dotsm u_2' u_1'=s_i^*e_i^*(\uu)$.
\begin{enumerate}
\item We have $u'_j=u_j$ for $j\notin \{i,i+1\}$ and
$u'_{i+1}u'_i \equiv u_{i+1} u_i$. 
\item Say $u_{i+1}$ and $u_i$ have lengths $b$ and $a$ respectively.
Then $(u'_{i+1},u'_i)$ is the unique pair of row words such that $u'_{i+1}u'_i\equiv u_{i+1}u_i$ such that $u'_{i+1}$ has length $a+1$
and $u'_i$ has length $b-1$.
\end{enumerate}
\end{prop}

\begin{ex} \label{X:jeu} Let $\uu=(u_2,u_1)$ with $u_1=11334$ and $u_2 = 2234$.
The rows have sizes $(5,4)$.
We compute $s_1^*e_1^*(\uu)=\tuu=(\tu_2 ,\tu_1)$.
The new rows should have sizes $(3,6)$.
This is computed by the two-row skew tableau jeu-de-taquin, which is known to preserve Knuth equivalence \cite{LS}.
To move one number from the top row to the bottom row,
we put a hole after the end of the bottom row and swap it left and up
while preserving semistandardness.
The exchange path of the hole is highlighted.
$$
\tableau[scy]{\bl&\tf{1}&\tf{1}&\tf{3}&3&4\\2&2&3&\tf{4}&\tf{\bullet}}
\longrightarrow
\tableau[scy]{\bl&\bl&1&1&3&4\\2&2&3&3&4}
$$
We do it again:
$$
\tableau[scy]{\bl&\bl&\tf{1}&\tf{1}&\tf{3}&\tf{4}\\2&2&3&3&4&\bullet}
\longrightarrow
\tableau[scy]{\bl&\bl&\bl&1&1&3\\2&2&3&3&4&4}
$$
The result is $s_1^*e_1^*\uu$ with $\tu_2=223344$ and $\tu_1=113$.
\end{ex}

\end{document}